\documentclass[letterpaper, 10 pt, conference]{ieeeconf}
\IEEEoverridecommandlockouts
\usepackage{amsmath,dsfont,amssymb}
\usepackage{amsthm}
\usepackage{mathtools}
\usepackage{graphicx}
\usepackage{amsfonts}
\usepackage{float}
\usepackage{mathrsfs}

\usepackage[hidelinks]{hyperref}
\usepackage{algorithmicx}
\usepackage{algorithm}
\usepackage{algpseudocode}
\usepackage{lipsum}
\usepackage{subfig}
\usepackage{amssymb}
\usepackage{bbm}
\usepackage{float}
\usepackage{balance}
\usepackage{color}
\usepackage{cleveref}
\theoremstyle{plain}
\newtheorem{thm}{\textbf{Theorem}}
\newtheorem{lem}{\textbf{Lemma}}
\newtheorem{prop}{\textbf{Proposition}}

\theoremstyle{definition}
\newtheorem{defn}{\textbf{Definition}}

\theoremstyle{remark}
\newtheorem{rem}{\textbf{Remark}}

\allowdisplaybreaks
\usepackage{algpseudocode}
\algdef{SE}[DOWHILE]{Do}{DoWhile}{\algorithmicdo}[1]{\algorithmicwhile\ #1}%

\newcommand{\prF}[2]{\langle #1,#2\rangle_\mathcal{F}}
\newcommand{\pref}{p^\mathrm{g,n}}
\newcommand{\pgen}{p^\mathrm{g}}

\usepackage[normalem]{ulem}

\usepackage{color}

\begin{document}

\title{\LARGE \bf
Distributionally Robust Decision Making Leveraging Conditional Distributions
}
\author{Yuxiao Chen, Jip Kim, and James Anderson
\thanks{Yuxiao Chen is with Nvidia Corporation, Santa Clara, CA, 95051, Email: {\tt yuxiaoc@nvidia.com}. Jip Kim and James Anderson are with Columbia University, New York, NY, 10027, Emails: {\tt \{jk4564,ja3451\}@columbia.edu}}
%\thanks{Jip Kim is partially funded by U. S. Department of Energy, Office of Energy Efficiency and Renewable Energy, Solar Energy Technologies Office under DE-EE0008769. James Anderson is partially funded through the NSF award CAREER 2144634.}
}

\maketitle
\thispagestyle{empty}

%\IEEEtitleabstractindextext{
\begin{abstract}
Distributionally robust optimization (DRO) is a powerful tool for decision making under uncertainty. It is particularly appealing because of its ability to leverage existing data. However, many practical problems call for decision-making with some auxiliary information, and DRO in the context of conditional distribution is not straightforward. We propose a conditional kernel distributionally robust optimization (CKDRO) method that enables robust decision making under conditional distributions through kernel DRO and the conditional mean operator in the reproducing kernel Hilbert space (RKHS). In particular, we consider problems where there is a correlation between the unknown variable $y$ and an auxiliary observable variable $x$. Given past data of the two variables and a queried auxiliary variable, CKDRO represents the conditional distribution $\mathbb{P}(y|x)$ as the conditional mean operator in the RKHS space and quantifies the ambiguity set in the RKHS as well, which depends on the size of the dataset as well as the query point. To justify the use of RKHS, we demonstrate that the ambiguity set defined in RKHS can be viewed as a ball under a metric that is similar to the Wasserstein metric. The DRO is then dualized and solved via a finite dimensional convex program. The proposed CKDRO approach is applied to a generation scheduling problem and shows that the result of CKDRO is superior to common benchmarks in terms of quality and robustness.
\end{abstract}

\section{Introduction}\label{sec:intro}

Distributionally robust optimization (DRO) has attracted great attention recently as a framework for decision making under uncertainty. Under a classic stochastic decision making setup, the goal is to minimize the expected cost given a distribution of the uncertainty. In practice, this might be overly restrictive  as the uncertainty distribution is often not fully known. This concern motivates DRO, which seeks to minimize the worst case expected cost (maximize reward) under a set of distributions referred to as the ambiguity set. If the ambiguity set is large enough to contain the true distribution, the resulting decision provides an upper bound on achievable real-world performance. The simplest (and most conservative) ambiguity set contains any distribution on the support of the uncertainty; in which case  DRO reduces to a robust optimization \cite{ben2009robust}. Several popular constructions of the ambiguity set are based on metrics that measure the difference between probability distributions, such as the Wasserstein metric \cite{esfahani2018data,yang2020wasserstein}, $\phi$-divergence \cite{bayraksan2015data}, and moments \cite{delage2010distributionally}. See the review paper \cite{rahimian2019distributionally} for a comprehensive summary. DRO is also closely related to risk-averse optimization and planning \cite{ruszczynski2006optimization} as all coherent risk measures have a dual form that is a DRO problem over the corresponding ambiguity set.

Depending on the probability space, a DRO problem can be discrete (the random variable has discrete support), or continuous (the support is continuous). Computationally, the continuous problem is much more complex, and many methods rely on sampling to gain computation tractability \cite{esfahani2018data,yang2020wasserstein,coppens2021data}. Most existing data-driven methods assume i.i.d. sampling from the unknown distribution and establish the ambiguity set based on the empirical distribution of the samples. An i.i.d. assumption then provides a probabilistic guarantee that the true distribution will be contained inside an ambiguity set with high probability via concentration inequalities \cite{fournier2015rate,tropp2015introduction}.

Unfortunately, the i.i.d. condition does not hold for decision making problems with conditional distributions. To motivate this formulation, consider the following two examples:
\begin{itemize}
    % \item Power generation plan \yc{need some reference here}: the power generation needs to be planned before the load is observed. We know that there exist strong correlation between load and signals such as weather and time, and some history data can be leveraged to make an informed guess of the load profile.
    % \item Generation scheduling in power systems operation is scheduled with electricity load and renewable generation forecasts that contain inevitable uncertainty \cite{dall2016optimal,zhang2011chance}. Exploiting the existing strong correlation between the load and signals such as weather and time, and historical data, the power grid operator can make the more reliable and cost-efficient scheduling. \ja{[Jip[can you improve this bullet point]}
    \item Generation scheduling \cite{dall2016optimal,zhang2011chance}: to provide electricity at a minimal cost, generation scheduling needs to be robustified against uncertainties in power demand, which is dominated by ambient factors such as weather conditions. A power system operator can exploit historical data and take account for the correlation between the electricity demand and weather forecast to make more reliable and cost-efficient generation scheduling.
    \item Autonomous driving \cite{chen2018modelling,chen2020counter,chen2020reactive}: to plan safe motions for the autonomous vehicle, we need to reason about the surrounding human drivers' actions, which strongly depend on the scenario. Given past data of how drivers react to different scenarios, for a particular scenario, the motion planning of the autonomous vehicle needs to leverage the distribution of possible actions of surrounding human drivers conditioned on the scenario.
\end{itemize}
Such problems can be formulated as DRO under ambiguity sets defined on some conditional distributions, yet it is unclear how to quantify the ambiguity set about the conditional distribution and how to leverage past data as the i.i.d. condition does not hold. In such problems, both the conditional mean and uncertainty characterization is critical, motivating methods like quantile regression \cite{meinshausen2006quantile}. A two step process can be applied where the first step uses existing regression tools to predict the distribution of the uncertainty given the auxiliary variables, followed by a subsequent stochastic optimization step given the prediction \cite{bertsimas2020predictive}. The difficulties of DRO with conditional distributions are: (i) The conditional distribution is hard to characterize from data, especially in continuous domains, (ii) It is difficult to determine the size of the ambiguity set for conditional distributions.

The proposed CKDRO method is inspired by two existing tools. First, the kernel embedding method in a reproducing kernel Hilbert space (RKHS) provides a convenient characterization of the conditional distribution with conditional mean operators, and its empirical version can be computed directly from data samples. For reference, see the  papers by Song \textit{et al.} \cite{song2009hilbert}, Fukumizu \textit{et al.} \cite{fukumizu2004dimensionality}, and the review paper \cite{muandet2016kernel}. On a different front, DRO with an RKHS was recently studied in \cite{zhu2020kernel} where the authors propose to represent ambiguity sets with mean operator norms defined on an RKHS. Based on these ideas, we propose conditional kernel distributionally robust optimization (CKDRO) with the following contributions: (i) We extend kernel DRO to problems with conditional distributions for data-driven robust conditional decision making, (ii) We establish connections between ambiguity set in RKHS and the Wasserstein metric providing a better understanding the kernel DRO setup, (iii) We propose criteria to determine the ambiguity set for conditional DRO, (iv) We apply the CKDRO method on an optimal power flow (OPF) problem and showcase its benefit over benchmarks.

 \Cref{sec:prelim} reviews some key concepts and tools, \cref{sec:CKDRO} presents the CKDRO method and establishes the connection to DRO through the Wasserstein metric, and \cref{sec:OPF} presents the application of CKDRO on an OPF problem. %and finally we conclude in \cref{sec:conclusion}

\textit{Nomenclature:}
$\mathbb{R}_+$ denotes the non-negative real line, $\langle \cdot,\cdot \rangle$ denotes the inner product of two vectors ($\langle x,y \rangle=x^\intercal y$) or two functions ($\langle f,g \rangle=\int f(x)g(x)dx $). For a random variable $X$, we use the calligraphic $\mathcal{X}$ denote its domain, $\mathbb{E}_Q[f(X)]$ denotes the expectation of a function $f:\mathcal{X}\to\mathbb{R}^n$ under the probability distribution $Q$. When $X$ has a fixed distribution, e.g., it's sampled from a given distribution, we simply write $\mathbb{E}_X$ for the expectation. $\mathcal{M}(\mathcal{X})$ denotes the space of all probability distributions $Q$ supported on $\mathcal{X}$ with $\mathbb{E}_Q[||X||]<\infty$. All vector norms are $\ell_2$-norms unless otherwise specified.

\section{Preliminaries}\label{sec:prelim}
%We start by reviewing some preliminaries.

\subsection{Hilbert space embedding}
Consider a random variable $X\in\mathcal{X}$. $X$ is endowed with some $\sigma$-algebra $\mathcal{A}$, and $\mathcal{P}$ denotes the space of probability distributions over $X$. A \textit{reproducing kernel Hilbert space} (RKHS) $\mathcal{F}$ on $\mathcal{X}$ with a symmetric kernel $k:\mathcal{X}\times\mathcal{X}\to\mathbb{R}$ is a Hilbert space of functions $f:\mathcal{X}\to\mathbb{R}$ whose inner product $\prF{\cdot}{\cdot}$ satisfies the reproducing property:
%\begin{subequations*}%\label{eq:reprod}
  \begin{align*}
    \prF{f(\cdot)}{k(x,\cdot)} &= f(x), \\
    \prF{k(x,\cdot)}{k(x',\cdot)} &= k(x,x').
  \end{align*}
%\end{subequations*}
The feature map $k(x,\cdot)$ is also denoted as $\varphi(x)$. A kernel $k$ is \textit{positive definite} if $\forall x_1,...x_N\in\mathcal{X}, \forall c_1,...c_N\in\mathbb{R}$, $\sum_{i,j=1}^N c_ic_jk(x_i,x_j)\ge 0$, or in  matrix form, $K\succeq 0$, where $K_{ij}=k(x_i,x_j)$. Common choice of kernels include the polynomial kernel $k(x,x')=(x^\intercal x')^d$, the Gaussian RBF kernel $k(x,x')=\exp(-\lambda||x-x'||^2)$, and the Laplacian kernel $k(x,x')=\exp(-\lambda||x-x'||_1)$. For every positive definite kernel $k$, there exists a unique RKHS with $k$ as its kernel up to isometry, and conversely, for every RKHS, its kernel is unique and positive definite \cite{muandet2016kernel}. A kernel is called \textit{translation invariant} if $\forall x_1,x_2,x'_1,x'_2\in\mathcal{X}$ with $x_1-x_2=x'_1-x'_2$, $k(x_1,x_2)=k(x'_1,x'_2)$. Both the Gaussian kernel and the Laplacian kernel are translation invariant, while the polynomial kernel is not.

The distance between points in an RKHS is also defined via the kernel: $d(x,x')=\sqrt{k(x,x)-2k(x,x')+k(x',x')}$.

So far, an RKHS is not related to any probability distribution. The central concept of kernel embedding is the expectation operator:
\begin{equation*}%\label{eq:mu}
  \mu_Q := \mathbb{E}_Q[\varphi(X)],
\end{equation*}
which is the expectation of the feature map under distribution $Q$. When the context is clear, we also use $\mu_X$ to denote the expectation operator when $X$ follows a fixed distribution. Assuming $\mathbb{E}_X[k(X,X)]<\infty$, $\mu_X$ is itself an element of $\mathcal{F}$. Given the reproducing property, we have $\forall f\in\mathcal{F},~\prF{\mu_X}{f}=\mathbb{E}_X[\prF{\varphi(X)}{f(X)}]=\mathbb{E}_X[f(X)]$. Its empirical estimate given a set of samples is simply
\begin{equation*}
  \hat{\mu}_X:=\frac{1}{N}\sum_{i=1}^{N}\varphi(x_i),
\end{equation*}
where $\{x_i\}_{i=1}^N$ is the sample set, assumed to have been drawn i.i.d. from $P$. The norm on $\mathcal{F}$ is defined as:
\begin{equation*}
  ||f||_\mathcal{F}^2=\langle f,f\rangle_\mathcal{F}.
\end{equation*}
If $f$ has a discrete representation: $f=\sum_i^N \alpha_i \varphi(x_i)$, then by the reproducing property,
\begin{equation*}
  ||f||_\mathcal{F}^2=\sum_{i,j=1}^N \alpha_i\alpha_j k(x_i,x_j)=\alpha^\intercal K \alpha.
\end{equation*}

Now consider a second random variable $Y$ with domain $\mathcal{Y}$, and let $\mathcal{B}$ and $\mathcal{Q}$ denote the $\sigma$-algebra and the space of all probability distributions over $Y$ respectively. An RKHS $\mathcal{G}$ is define on $\mathcal{Y}$, induced by the positive definite kernel $l:\mathcal{Y}\times\mathcal{Y}\to\mathbb{R}$. Let $\phi(y)=l(y,\cdot)$ be its feature map, the (uncentered) covariance operator is defined as
\begin{equation*}%\label{eq:cov}
  C_{YX}:=\mathbb{E}_{YX}[\phi(Y)\otimes\varphi(X)],
\end{equation*}
where $\otimes$ is the tensor product. $C_{YX}$ can be viewed as an operator from $\mathcal{F}$ to $\mathcal{G}$ in the sense that for any $f\in\mathcal{F},\forall g\in\mathcal{G}$, $\langle g, C_{YX} f\rangle_\mathcal{G} = Cov[f(X),g(Y)]$. Given samples $\{(x_i,y_i)\}_{i=1}^N$, the empirical evaluation of the covariance operator is
\begin{equation*}%\label{eq:empi_cov}
  \hat{C}_{YX}=\frac{1}{N}\Upsilon \Phi^\intercal
\end{equation*}
where $\Upsilon=[\varphi(x_1),\varphi(x_2)...,\varphi(x_N)]$ is the feature matrix for $X$, $\Phi=[\phi(y_1),\phi(y_2)...,\phi(y_N)]$ is the feature matrix for $Y$.
\begin{lem}[\cite{fukumizu2004dimensionality}]
  If $\mathbb{E}_{YX}[g(Y)|X=\cdot]\in\mathcal{G}$ for any $g\in\mathcal{G}$, then
  \begin{equation*}
    C_{XX}\mathbb{E}_{YX}[g(Y)|X=\cdot]=C_{YX} g.
  \end{equation*}
\end{lem}
Consequently, \cite{song2009hilbert} shows that if $C_{XX}$ is invertible, then we can define an operator $\mathcal{U}_{Y|X}: \mathcal{F}\to\mathcal{G}$ where $\forall x\in\mathcal{X},$ $\mu_{Y|x}:=\mathcal{U}_{Y|X}\varphi(x)$, and $\mathcal{U}_{Y|X}$ is defined as
\begin{equation*}%\label{eq:cond_mu}
  \mathcal{U}_{Y|X} = C_{XX}^{-1}C_{YX}.
\end{equation*}
Similar to the unconditioned expectation operator, $\mu_{Y|x}$ can be evaluated empirically with samples.
\begin{lem}[\cite{song2009hilbert}, Theorem 5]
  Let $k_x:=\Upsilon^\intercal\varphi(x)$, then $\hat{\mu}_{Y|x}$ is the empirical conditioned expectation operator, and is estimated as $\hat{\mu}_{Y|x}=\Phi(K+\lambda N I_N)^{-1} k_x$, where $\lambda>0$ is the regularization parameter that prevents overfitting.
\end{lem}
\subsection{Distributionally robust optimization}
Distributionally robust optimization considers the following problem:
\begin{equation}\label{eq:DRO}
  \min_{\eta}\max_{Q\in\mathscr{A}} \mathbb{E}_{Q} [q(\eta,X)],
\end{equation}
where $\eta\in H$ is the decision variable, $X$ is the random variable, $\mathscr{A}\subseteq\mathcal{P}$ is the ambiguity set that contains all the distribution $Q$ over $X$ under consideration. As mentioned in the introduction, the ambiguity set is typically defined with some metric that measures the difference between distributions, such as the Wasserstein metric or  $\phi$-divergence. In particular, the Wasserstein metric is defined on $\mathcal{M}(\mathcal{X})$ as follows.
\begin{defn}[Wasserstein metric]\label{def:wass}
  The Wasserstein metric $d_W:\mathcal{M}(\mathcal{X})\times\mathcal{M}(\mathcal{X})\to\mathbb{R}_+$ is defined as
  %\begin{equation*}
    \begin{align*}
    d_W(Q_1,Q_2)& :=\inf_{\Pi}\int_{\mathcal{X}^2}||x_1-x_2||\Pi(x_1,x_2) dx_1 dx_2\\
    \mathrm{s.t.}~& \int_\mathcal{X} \Pi(x_1,x_2) dx_2= Q_1(x_1),\\
     &\int_\mathcal{X} \Pi(x_1,x_2) dx_1= Q_2(x_2).
    \end{align*}
  %\end{equation*}
\end{defn}
The Wasserstein metric is essentially defined by the optimal transport problem, where $\Pi$ is a joint distribution supported on $\mathcal{X}^2$ with the two marginal distributions being $Q_1$ and $Q_2$.

An important equivalent definition was proposed in \cite{kantorovich1958space}, which is used later in this paper.
\begin{lem}[Kantorovich-Rubinstein \cite{kantorovich1958space}]\label{lem:kanto}
  \begin{equation*}%\label{eq:KR}
    d_W(Q_1,Q_2) = \sup_{L(f)\le 1}\left\{\int_\mathcal{X} f(x)Q_1(dx)-\int_\mathcal{X} f(x)Q_2(dx)\right\},
  \end{equation*}
  where $L(f)$ denotes the Lipschitz constant of the function $f$.
\end{lem}
Ambiguity sets  can then be equivalently defined as $\mathscr{A}=\{Q~|~d_W(Q_0,Q)<\epsilon\}$, where $Q_0$ is the nominal distribution, which can be given, or estimated from data. Work in \cite{esfahani2018data} presented a convex program that solves the DRO with the above ambiguity set where $Q_0$ is the empirical distribution of the samples $Q_0=\frac{1}{N}\delta(x_i)$, where $\delta(\cdot)$ is the Dirac measure.

Kernel DRO was recently proposed in \cite{zhu2020kernel} as a convenient data-driven DRO framework that leverages the RKHS. The distributions are represented with kernel embeddings in an RKHS, and the ambiguity set is typically chosen as an RKHS ball: $\mathcal{C}:=\{\mu~|~||\mu-\mu_0||_\mathcal{F}\le \epsilon\}$. For clarity, we use $\mathcal{C}$ to denote an ambiguity set defined with respect to the expectation operator, and $\mathscr{A}$ to denote an ambiguity set defined by a probability distribution. They are linked by: $\mathscr{A}=\{P\in\mathcal{P}~|~\int_\mathcal{X}\varphi(x)P(dx)\in\mathcal{C}\}$.

There are two common strategies for solving the DRO problem~\eqref{eq:DRO}; the cutting plane technique \cite{calafiore2007ambiguous} and the dual formulation, similar to the ones used in robust optimization \cite{ben2009robust}. When the inner maximization problem is a  concave problem, the dual formulation offers a  convex program for the whole DRO.

For clarity, we shall leave the solution of the kernel DRO to the next section, and present it together with the conditional kernel DRO.

\section{Conditional kernel DRO}\label{sec:CKDRO}
In this section, we present the formulation of the conditional kernel DRO.
\subsection{Problem formulation}
Suppose we are given a training data set $\{x_i,y_i\}_{i=1}^N$ consisting of $X,Y$ pairs, where $X$ is an auxiliary variable and the distribution of $X$ and $Y$ are correlated. The conditional DRO problem then tries to minimize the worst case expectation of the cost given an realization of $X$. In kernel DRO, the ambiguity set is defined by the expectation operator $\mu_{Y|x}$, and the DRO is defined as
\begin{equation*}\label{eq:cond_DRO}
  \min_{\eta}\max_{\mu\in\mathcal{C}[x]} \langle \mu,q(\eta,\cdot)\rangle,
\end{equation*}
where $\mathcal{C}[x]\subseteq\mathcal{G}$ is an ambiguity set on the conditional expectation operator $\mu_{Y|X}$ at $X=x$, and $q:H\times \mathcal{Y}\to \mathbb{R}$ is the cost function.

\subsection{Construction of the ambiguity set}
We first discuss the construction of the ambiguity set $\mathcal{C}[x]$. Two sources of uncertainty about $\mu_{Y|x}$ are considered. The first  is due to the estimation of $\mu_{Y|x}$ from data, the second  is due to  unseen data.
\begin{lem}[Theorem 6 in \cite{song2009hilbert}]\label{lem:conv}
  Assuming $\varphi(x)$ is in the range of $\mathcal{C}_{XX}$, $||\hat{\mu}_{Y|x}-\mu_{Y|x}||_{\mathcal{G}}$ converges to zero with rate $O_p(\lambda^\frac{1}{2}+(N\lambda)^{-\frac{1}{2}})$. \footnote{$O_p$
 stands for probabilistic convergence rate.}
 \end{lem}
Faster convergence rates are available under additional assumptions about the underlying distribution in \cite{grunewalder2012conditional}, but for simplicity, we stick to the convergence rate provided in Lemma \ref{lem:conv}. If we choose $\lambda\sim N^{-\frac{1}{4}}$, the convergence rate is $O_p(N^{-\frac{1}{4}})$.

Note that the above convergence result is on $\mu_{Y|X}$, not on any particular $\mu_{Y|x}$, where the former is an average of the later weighted by $P$, the distribution of $X$. A naive choice of the ambiguity set would be $\mathscr{A}=\{\mu~|~||\mu-\hat{\mu}_{Y|x}||_\mathcal{G}\le \gamma N^{-\frac{1}{4}}\}$, where $\hat{\mu}_{Y|x}$ is the empirical conditional expectation operator evaluated at $X=x$, $N$ is the number of data points in the training set. However, note that the size of the ambiguity set does not change with $x$, indicating that the uncertainty on $\mu_{Y|x}$ is the same whether the queried $x$ is close to the samples or far away from the samples.

To account for the difference in queried $x$, we propose to use a fictitious sample at $X=x$ to help measure the uncertainty. Now suppose that a fictitious sample of $Y$ drawn from $\mathbb{P}(Y|X=x)$ is added to the samples, denoted as $[x,y_x]$. The empirical evaluation of the conditional mean operator $\mu_{Y|x}$ with the augmented sample set is then
\begin{equation*}
  \hat{\mu}_{Y|x} = {k^A_x}^\intercal (K^A+\lambda (N+1) I_{N+1})^{-1}\Phi^A,
\end{equation*}
where $A$ stands for ``augmented'',
\begin{align*}
k^A_x &=[k(x_1,x),\hdots,k(x_N,x),k(x,x)]^\intercal   \\ \Phi^A &=[\phi(y_1),\phi(y_2),\hdots,\phi(y_N),\phi(y_x)]^\intercal
\end{align*}
and $K^A=\begin{bmatrix}
                                                                                                                                                 K & k_x \\
                                                                                                                                                 k_x^\intercal & k(x,x)
                                                                                                                                               \end{bmatrix}$.
Let
\begin{equation}\label{eq:beta}
  \beta:={k^A_x}^\intercal (K^A+\lambda (N+1) I_{N+1})^{-1},
\end{equation}
then $\hat{\mu}_{Y|x}=\sum_{i=1}^N \beta_i \phi(y_i)+\beta_{N+1} \phi(y_x)$.

\begin{prop}\label{prop:amb}
 Suppose for all $y,y'\in\mathcal{Y},l(y,y')\le R$, the ambiguity set of $\hat{\mu}_{Y|x}$ can be written as $\mathcal{C}[x]=\{\mu~|~||\mu-\sum_{i=1}^{N}\beta_i \phi(y_i)||_\mathcal{G}\le |\beta_{N+1}|R+\gamma (N+1)^{-\frac{1}{4}}\}$ with $\beta$ defined in~\eqref{eq:beta}.
\end{prop}
\begin{proof}
  By the triangular inequality, $||\mu_{Y|x}-\sum_{i=1}^{N}\beta_i \phi(y_i)||_\mathcal{G}\le ||\mu_{Y|x}-\hat{\mu}_{Y|x}||_\mathcal{G}+||\beta_{N+1}\phi(y_x)||_\mathcal{G}$, where the first term is bounded by $|\beta_{N+1}|R$, and the second term is bounded by $\gamma (N+1)^{-\frac{1}{4}}$ according to Lemma \ref{lem:conv}.
\end{proof}
The motivation for this formulation is to use the fictitious sample as a probe to measure how a newly added sample at $X=x$ would change the conditional expectation. For a queried $x$ that is close to the samples (as measured by the kernel), $|\beta_{N+1}|$ is small and the uncertainty is small; for an $x$ far away from the samples, the conditional expectation depends heavily on the fictitious sample, which is unknown, thus enlarging the ambiguity set.
\subsection{Connection to Wasserstein ambiguity set}

To better motivate the use of the RKHS norm as a metric for the ambiguity set definition, we establish  connections between the ambiguity set defined with the RKHS norm and the Wasserstein ball. Specifically, by Lemma \ref{lem:kanto}, the Wasserstein metric between two distributions can be characterized by some test function $f$ with Lipschitz constant bounded by 1.
\begin{lem}
  Given a translation invariant kernel $k$, suppose there exists a constant $\mathtt{L}$ such that $\forall x_1,x_2\in\mathcal{X}$, $d(x_1,x_2)\le \mathtt{L} ||x_1-x_2||$, then for any $f\in\mathcal{F}$ with a bounded RKHS norm, $f$ is Lipschitz modulo $\mathtt{L}||f||_\mathcal{F} $.
\end{lem}
\begin{proof}
\begin{equation*}
  \begin{aligned}
  f(x_1)-f(x_2)&= \prF{f}{k(\cdot,x_1)}-\prF{f}{k(\cdot,x_2)}\\
  &=\prF{f}{k(\cdot,x_1)-k(\cdot,x_2)}\\
  &\le ||f||_\mathcal{F}d(x_1,x_2)\\
  &\le \mathtt{L}||f||_\mathcal{F}  ||x_1-x_2||.
  \end{aligned}
\end{equation*}
  where the first inequality is due to the Cauchy-Schwartz inequality.
\end{proof}
The constant $\mathtt{L}$ varies with the kernel $k$, for a Gaussian kernel $k(x_1,x_2)=\exp(-\frac{||x_1-x_2||^2}{2\sigma^2})$, $\mathtt{L}$ is simply $\frac{1}{\sigma}$.
\begin{prop}
  Consider all distributions $Q$ satisfying $\mathbb{E}_Q[k(x,x)]<\infty$ for all $x\in\mathcal{X}$, and the following two ambiguity sets:
  \begin{equation*}
    \begin{aligned}
    \mathscr{A}_1&=\{Q~|~d_W(Q,Q_0)||\le \epsilon\}\\
    \mathscr{A}_2&=\{Q~|~||\mu_Q-\mu_{Q_0}||_\mathcal{F}\le \epsilon\mathtt{L}\},
    \end{aligned}
  \end{equation*}
$\mathscr{A}_2\subseteq\mathscr{A}_1$.
\end{prop}
\begin{proof}
  By Lemma \ref{lem:kanto},
  \begin{equation*}
  \resizebox{1\columnwidth}{!}{$
    \begin{aligned}
    ||\mu_Q-\mu_{Q_0}||_\mathcal{F}&=\mathtt{L}\sup_{f\in\mathcal{F},||f||_\mathcal{F}\le \frac{1}{\mathtt{L}}}\langle \mu_Q-\mu_{Q_0},f\rangle\\
    &= \mathtt{L} \sup_{f\in\mathcal{F},||f||_\mathcal{F}\le \frac{1}{\mathtt{L}}}\{\int_\mathcal{X} f(x)Q_1(dx)-\int_\mathcal{X} f(x)Q_2(dx)\}\\
    &\le \mathtt{L} \sup_{L(f)\le 1}\{\int_\mathcal{X} f(x)Q_1(dx)-\int_\mathcal{X} f(x)Q_2(dx)\}\\
    &=\mathtt{L} d_W(Q,Q_0).
    \end{aligned}
    $}
  \end{equation*}
For any $Q\in\mathscr{A}_1$, we have $||\mu_Q-\mu_{Q_0}||_\mathcal{F}\le \mathtt{L} d_W(Q,Q_0)\le \mathtt{L} \epsilon$.
\end{proof}
The ambiguity set defined with the RKHS norm can be viewed as a subset of the ambiguity set defined via the Wasserstein metric. In fact, $\mathscr{A}_2$ can be viewed as the Wasserstein ball under a different distance measure.
\begin{prop}\label{prop:wass_RKHS}
  Suppose the Wasserstein metric is defined as
  \begin{equation*}
    \begin{aligned}
    \bar{d}_W(Q_1,Q_2)& :=\inf_{\Pi}\int_{\mathcal{X}^2}d(x_1,x_2)\Pi(x_1,x_2) dx_1 dx_2\\
    \mathrm{s.t.}~& \int_\mathcal{X} \Pi(x_1,x_2) dx_2= Q_1(x_1),\\
     &\int_\mathcal{X} \Pi(x_1,x_2) dx_1= Q_2(x_2),
    \end{aligned}
  \end{equation*}
  where $d(x_1,x_2)$ is the distance metric in an RKHS $\mathcal{F}$, then the dual form of the Wasserstein metric is
  \begin{equation*}
    \bar{d}_W(Q_1,Q_2) = \sup_{||f||_\mathcal{F}\le 1}\{\int_\mathcal{X} f(x)Q_1(dx)-\int_\mathcal{X} f(x)Q_2(dx)\}.
  \end{equation*}
\end{prop}
\begin{proof}
  The proof follows exactly the Kantorovich Rubinstein duality theorem \cite{edwards2011kantorovich} after replacing the Euclidean distance with the RKHS distance metric $d$, and is omitted here.
\end{proof}
By Proposition \ref{prop:wass_RKHS}, the ambiguity set defined as an RKHS norm ball is simply a Wasserstein ball with $d$ as the distance metric. For a Gaussian RBF kernel, the relationship between $d$ and the Euclidean distance is plotted in \cref{fig:RBF}. Comparing to the original Wasserstein ball in \ref{def:wass}, the RKHS distance ``saturates'' at $\sqrt{2}$, which discounts the distance when two points are far away.
\begin{figure}
  \centering
  \includegraphics[width=0.8\columnwidth]{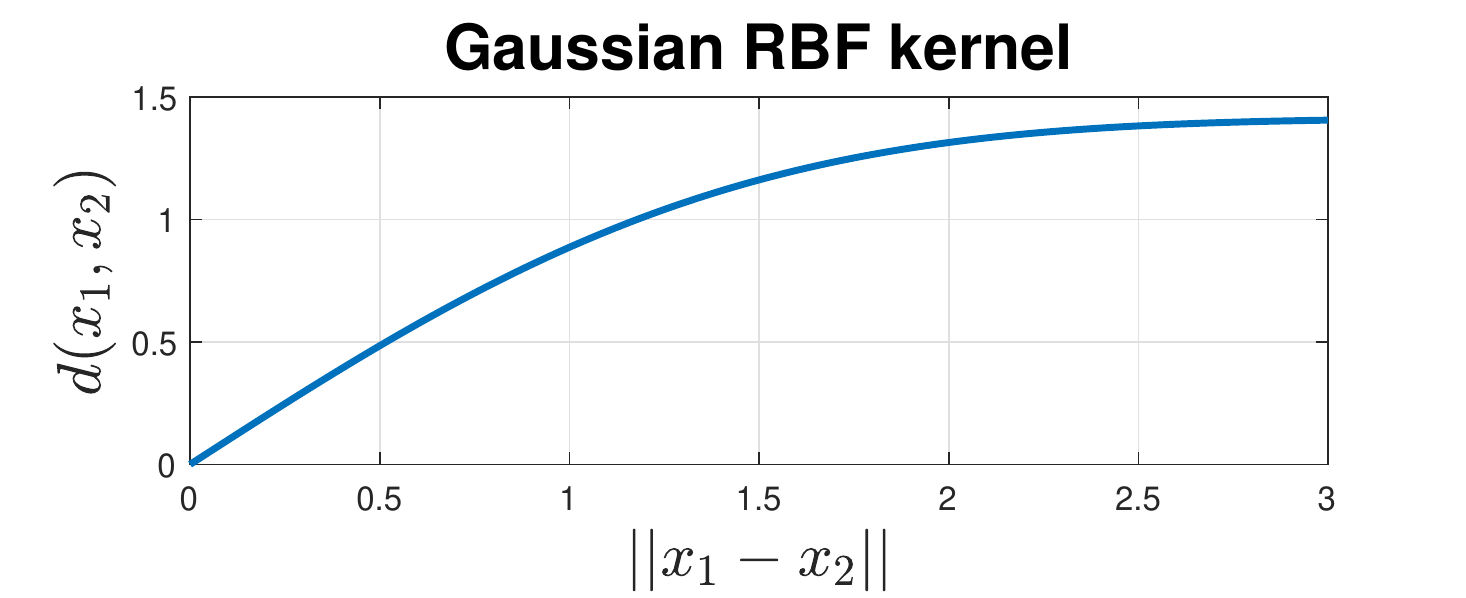}
  \caption{RKHS distance v.s. Euclidean distance with Gaussian RBF kernel}\label{fig:RBF}
\end{figure}
\subsection{Solution to the conditional kernel DRO}
For brevity of notation, given the samples $\{x_i,y_i\}_{i=1}^N$ and a particular $x\in\mathcal{X}$, let $\hat{\mu}[x]:=\sum_{i=1}^N \beta_i \phi(y_i)$ with $\beta$ defined in~\eqref{eq:beta}, and $\mathcal{C}[x]=\{\mu\in\mathcal{G}~|~||\mu-\hat{\mu}[x]||_\mathcal{G}\le \epsilon_x\}$ be the ambiguity set defined in Proposition \ref{prop:amb}. The conditional kernel DRO is the following optimization:
\begin{equation}\label{eq:ckdro}
\begin{aligned}
\min_\eta \max_{Q\in co(\mathcal{Q}),\mu\in\mathcal{C}[x]} &~\int_\mathcal{Y} q(\eta,y)Q(dy)\\
\mathrm{s.t.} &~\int_{\mathcal{X}} 1 Q(dy)=1,\quad \int_{\mathcal{X}} \phi(y) Q(dy)=\mu,
\end{aligned}
\end{equation}
where $co(\cdot)$ is the conic hull. The inner maximization looks for a distribution $Q$ whose corresponding expectation operator $\mu$ lies inside $\mathcal{C}[x]$ and maximizes the expected cost.
\begin{thm}\label{thm:dual_CDRO}
  The dual form of~\eqref{eq:ckdro} is
\begin{equation}\label{eq:dual_CDRO}
  \begin{aligned}
  \min_{\eta,f\in\mathcal{G},f_0\in \mathbb{R}} ~&f_0 + \sum_{i=1}^N \beta_i f(y_i)+\epsilon_x ||f||_\mathcal{G}\\
  \mathrm{s.t.}~&\forall y\in\mathcal{Y}, q(\eta,y)\le f(y)+f_0.
  \end{aligned}
\end{equation}
\end{thm}
\begin{proof}
  The dual form was derived in \cite{zhu2020kernel}, we extend it to the conditional distribution case and present the derivation for completeness. First focus on the inner maximization, the Lagrangian of the of which is
  \begin{equation*}
  \resizebox{1\columnwidth}{!}{$
  \begin{aligned}
  \mathcal{L}(\mu,P,f,f_0)&=\int_\mathcal{Y} q(\eta,y)Q(dy)- \delta_{\mathcal{C}[x]}(\mu)\\
   &+\langle f,\mu-\int_\mathcal{Y} \phi(y)Q(dy)\rangle + f_0(1-\int_\mathcal{Y}1Q(dy))\\
  & = f_0 + \langle f,\mu\rangle - \delta_{\mathcal{C}[x]}(\mu) + \int_\mathcal{Y} q(\eta,y)-f(y)-f_0 Q(dy),
  \end{aligned}
  $}
  \end{equation*}
   where $f_0\in\mathbb{R}$ is the Lagrange multiplier of the first constraint in \eqref{eq:ckdro} and $f\in\mathcal{G}$ is the Lagrange multiplier of the second constraint. $\delta_{\mathcal{C}[x]}$ is the indicator function of $\mathcal{C}[x]$, which is 0 if $\mu\in\mathcal{C}[x]$, $\infty$ otherwise. The only way to upper bound $\mathcal{L}$ is to make $q(\eta,y)-f(y)-f_0\le 0$ for all $y\in\mathcal{Y}$. With $f_0$ and $f$ fixed that satisfy $q(\eta,y)-f(y)-f_0$, the maximization is only over $\mu\in\mathcal{C}[x]$ of the function $\langle f,\mu\rangle - \delta_{\mathcal{C}[x]}(\mu)$. Since $\delta_{\mathcal{C}[x]}$ is convex, by Fenchel duality,
   \begin{equation*}
     \max_{\mu\in\mathcal{G}}  \langle f,\mu\rangle - \delta_{\mathcal{C}[x]}(\mu)=\delta^*_{\mathcal{C}[x]}(f)=\max_{\mu\in\mathcal{C}[x]} \langle f,\mu\rangle,
   \end{equation*}
   where $\delta^*_{\mathcal{C}[x]}$ is the Fenchel dual of $\delta_{\mathcal{C}[x]}$. By the definition of $\mathcal{C}[x]$ in Proposition \ref{prop:amb},
   \begin{align*}
     \delta^*_{\mathcal{C}[x]}(f)&=\max_{\mu\in\mathcal{C}[x]} \langle f,\mu\rangle=\max_{\mu\in\mathcal{C}[x]} \langle f,\hat{\mu}[x]\rangle + \langle f,\mu-\hat{\mu}[x]\rangle  \\
     & = \langle f,\hat{\mu}[x]\rangle + \epsilon_x||f||_\mathcal{G}.
   \end{align*}
  The last equality comes from the fact that we can choose $\mu$ such that $\mu-\hat{\mu}[x]=\epsilon_x \frac{f}{||f||_\mathcal{G}}$. The dual optimization is then
  \begin{equation*}
  \begin{aligned}
  \min_{f\in\mathcal{G},f_0\in\mathbb{R}}& f_0 + \langle f,\hat{\mu}[x]\rangle + \epsilon_x||f||_\mathcal{G}\\
  \mathrm{s.t.} & \forall y\in\mathcal{Y},q(\eta,y)\le f_0+f(y).
  \end{aligned}
  \end{equation*}
  Combining with the outer minimization, the dual form of \eqref{eq:ckdro} is the optimization in \eqref{eq:dual_CDRO}. Since the inner optimization is strictly feasible, strong duality holds, and there is no duality gap.
\end{proof}
Given an $x\in\mathcal{X}$, $\hat{\mu}[x]=\sum_{i=1}^N \beta_i \phi(y_i)$, by the reproducing property, $\langle f,\hat{\mu}[x]\rangle = \sum_{i=1}^N \beta_i f(y_i)$. However, the inequality constraint in \eqref{eq:dual_CDRO} cannot be strictly enforced. Instead, the condition is enforced on a finite certification set of $\{y_j\}_{j=1}^M$. $f$ is a function in RKHS, which is also finitely parameterized by the certification set as $f=\sum_{j=1}^M \alpha_j \phi(y_j)$. The finite optimization we end up solving is then
\begin{equation}\label{eq:dual_CDRO_emp}
  \begin{aligned}
  \min_{\eta,\alpha\in\mathbb{R}^M,f_0\in \mathbb{R}} ~&f_0 + \sum_{i=1}^N \beta_i f(y_i)+\epsilon_x ||f||_\mathcal{G}\\
  \mathrm{s.t.}~& q(\eta,y_j)\le f(y_j)+f_0,~j=1,...,M,
  \end{aligned}
\end{equation}
where $f(y)=\sum_{j=1}^M \alpha_j l(y_j,y)$, $||f||_\mathcal{G}=\sqrt{\alpha^\intercal L \alpha}$. Here $L\in\mathbb{R}^{M\times M}$ is computed as $L_{ij}=l(y_i,y_j)$.

\section{Application to Generation Scheduling}\label{sec:OPF}

\begin{figure}[th]
  \centering
  \includegraphics[width=1\columnwidth]{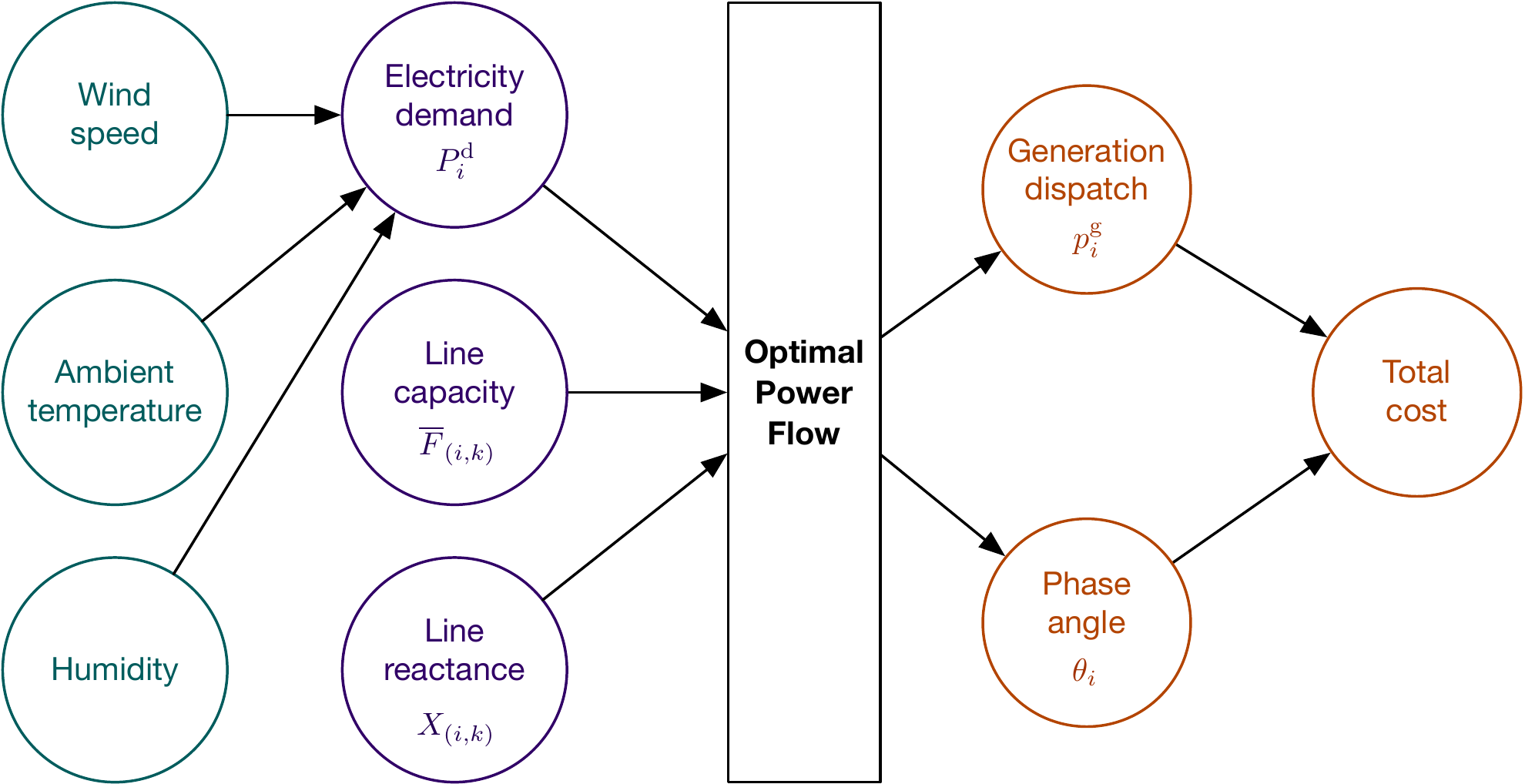}
  \caption{Causality map of the OPF problem}\label{fig:causality}
\end{figure}

As an application of the conditional kernel DRO, we consider the problem of scheduling power generation for power grid. We formulate this via a DC optimal power flow (DC-OPF) problem; a constrained optimization problem that determines generation dispatch to serve electricity demand in the transmission network at a minimal cost subject to operational constraints.
% This section reviews the optimal power flow (OPF) formulation and its linear approximation, so-called, DC-OPF which is used for the wholesale market settlements in practice.
\subsection{Background of DC-OPF}
The DC-OPF problem is formulated as:
\begin{subequations}\label{eq:OPF}
\begin{align}
  \min_{\pgen,\theta}\hspace{4mm}
      & \sum_{j\in{\cal G}}
        \left[
          a_j (\pgen_j)^2 + b_j \pgen_j + c_j
        \right]
      \label{eq:opf_obj}\\
  % \mathrm{s.t.} & \nonumber\\
  \mathrm{s.t.}\hspace{4mm}& \sum_{j\in{\cal G}_i}\pgen_j - P^{\mathrm{d}}_i = \sum_{k\in{\cal N}}
      \frac{\theta_{i} - \theta_{k}}{X_{ik}},\quad\forall i \in{\cal N},
      \label{eq:opf_pbalance}\\
  &\underline{P}^{\mathrm{g}}_{j} \leq \pgen_{j} \leq \overline{P}^{\mathrm{g}}_{j}, \quad\forall j \in \cal{G}, \label{eq:opf_gpbounds}\\
  &\underline{F}_{ik} \le \frac{1}{X_{ik}}(\theta_{i} - \theta_{k}) \le \overline{F}_{ik},\quad\forall i,k\in{\cal N}.\label{eq:opf_fbounds}
\end{align}
\end{subequations}
where the active power output of generator $j\in{\cal G}$ is denoted as $\pgen_j$ and the nodal phase angle of bus $i\in{\cal N}$ is denoted as $\theta_i$.
The objective function in \eqref{eq:opf_obj} minimizes the total quadratic generation cost with parameters $a_j$, $b_j$ and $c_j$.
Equation~\eqref{eq:opf_pbalance}, i.e. a linear approximation of the power flow equation, enforces the active nodal power balance where $Y_{ik}$ is the reactance of transmission line $(i,k)$.
Active power output of generators is constrained in~\eqref{eq:opf_gpbounds} using their lower and upper bound limits ($\underline{P}_j^{\mathrm{g}}$, $\overline{P}_j^{\mathrm{g}}$).
Lastly, the power flow limits of each transmission line are constrained with lower and upper bounds ($\underline{F}_{ik}$, $\overline{F}_{ik}$) in \eqref{eq:opf_fbounds}.
\begin{rem}
In practice, we add slack variables to relax the equality constraint \eqref{eq:opf_pbalance} to improve feasibility, but we found that the slack variable is always close to zero.
\end{rem}
\subsection{Conditional DC-OPF problem setup}
% Unfortunately, it is usually difficult to obtain the power demand $P^d$ at the time of OPF.
Unfortunately, the power demand $P^{\mathrm{d}}_i$ cannot be predicted perfectly at the time of OPF. To robustify the OPF solution against the prediction error, we consider a two-stage optimization where the decision variables are divided into two groups, the here-and-now (HAN) variables and the wait-and-see (WAS) variables \cite{wets2002stochastic}, where the HAN variables are determined in the first stage, the WAS variables can be determined after an observation of the uncertain parameters is made. Similar formulations have been applied to generation scheduling problems (c.f. \cite{zhang2019two,zhao2019two}), here we consider a simplified two-stage DC-OPF problem. The first stage determines the nominal generation dispatch $\pref$, then the second stage determine the actual generation by solving the following optimization:
\begin{equation}\label{eq:second_stage_OPF}
\begin{aligned}
   \mathrm{OPF}(P^d,\pref)=&\mathop{\min}\limits_{\theta,\pgen}
      \mathcal{J}_n(\pgen) + {\color{red}\mathcal{J}_a(\pref,\pgen)}\\
  \mathrm{s.t.}\; & {\color{red}\pgen\in \mathcal{F}(\pref)}\\
  & \sum_{j\in{\cal G}_i}\pgen_j - P^{\mathrm{d}}_i = \sum_{k\in{\cal N}}
      \frac{\theta_{i} - \theta_{k}}{X_{ik}},\forall i \in{\cal N}, \\
  &\underline{P}^{\mathrm{g}}_{j} \leq \pgen_{j} \leq \overline{P}^{\mathrm{g}}_{j}, \;\forall j \in \cal{G}, \\
  &\underline{F}_{ik} \le \frac{1}{X_{ik}}(\theta_{i} - \theta_{k}) \le \overline{F}_{ik},\;\forall i,k\in{\cal N},
\end{aligned}
\end{equation}
where $\pgen$ is the actual power generation. $\mathcal{J}_n(\pgen)=\sum_{j\in{\cal G}} a_j (\pgen_j)^2 + b_j \pgen_j + c_j$ is the generation cost, $\mathcal{J}_a$ is the adjusting cost, which penalizes the difference between $\pref$ and $\pgen$, and $\mathcal{F}(\pref)$ is the feasible set of the actual generation given the nominal generation dispatch. The HAN variable is the nominal generation dispatch $\pref$, the WAS variables are the actual generation dispatch $\pgen$ and phase angle $\theta$. The feasible set $\mathcal{F}$ is defined as $\mathcal{F}(\pref)=\{p\in \mathbb{R}^{|\mathcal{G}|}|\forall j \in \mathcal{G}\backslash\mathcal{G}_{\mathrm{s}}, p_j=\pref_j \}$, where $\mathcal{G}_{\mathrm{s}}$ is the set of flexibility resources, i.e. the generation dispatch can be adjusted after the nominal dispatch. Additional constraints can be added such as ramping constraints on the change of power generation. The adjusting cost then penalizes the adjustment of flexible generators $j\in\mathcal{G}_{\mathrm{s}}$. In particular, we consider the following adjusting cost function:
% \begin{equation}\label{eq:adjusting_cost}
%     \begin{aligned}
%     \mathcal{J}_a^1(\pref,\pgen) &=\sum_{j\in\mathcal{G}\backslash\mathcal{G}_s} \alpha_j (\pref_j-\pgen_j)^2\\
%     \mathcal{J}_a^1(\pref,\pgen) &=\sum_{j\in\mathcal{G}\backslash\mathcal{G}_s} \beta^+_j [\pref_j-\pgen_j]_+ + \beta^-_j [\pgen_j-\pref_j]_+,
%     \end{aligned}
% \end{equation}
\begin{equation*}%\label{eq:adjusting_cost}
   \mathcal{J}_a(\pref,\pgen) =\sum_{j\in\mathcal{G}\backslash\mathcal{G}_s} \zeta^+_j [\pref_j-\pgen_j]_+ + \zeta^-_j [\pgen_j-\pref_j]_+,
\end{equation*}
% where $\mathcal{J}_a^1$ is a symmetric quadratic cost that penalizes deviation of $\pgen$ from $\pref$;
which is an asymmetric linear cost that penalizes generation surplus and generation deficiency differently. $[\cdot]_+=\max\{\cdot,0\}$, and $\zeta^+$ is chosen to be smaller than $\zeta^-$ so that generation deficiency is more heavily penalized. Under the piecewise linear $\mathcal{J}_a$, the second stage is a convex optimization problem. $\mathrm{OPF}(P^d,\pref)$ denotes the optimal cost function of ~\eqref{eq:second_stage_OPF} whenever ~\eqref{eq:second_stage_OPF} is strictly feasible, and $\mathrm{OPF}(P^d,\pref)=\infty$ if~\eqref{eq:second_stage_OPF} is infeasible. Clearly, $\mathrm{OPF}(P^d,\pref)$ depends heavily on $P^d$, which is unknown at the first stage.

While the knowledge on $P^d$ is not perfect when deciding $\pref$ in the first stage, there is a causal relationship between the weather data and the electricity demand as shown in \cref{fig:causality}. Therefore, given historical data $\{\text{[weather,time]}, \text{demand}\}_{i=1}^N$, we build a conditional Kernel DRO for the DC-OPF problem to better capture the conditional distribution of $P^d$. The weather and time signal serves as the random variable being conditioned on ($X$), $P^d$ is the random variable whose conditional distribution we are interested in. The first stage decision making is then solved with the following DRO:

\begin{equation}\label{eq:DRO_OPF}
  \min_{\pref}\max_{Q\in\mathscr{A}_{\xi}} \mathbb{E}_Q[\mathrm{OPF}(P^d,\pref)],
\end{equation}
where $\xi$ is the auxiliary variable (weather, time, etc.), $\mathscr{A}_\xi$ is the ambiguity set over the conditional distribution of $P^d$ given $\xi$.

Following the conditional kernel DRO framework, the ambiguity set on conditional distribution is turned into an ambiguity set in the conditional mean operator in RKHS. Given samples $\{\xi_i,P^d_i\}_{i=1}^N$, we use Gaussian RBF kernel for both $\xi$ and $P^d$. In particular, each dimension of $\xi$ is equipped with its own kernel and the distances are added to form the total distance. For cyclic dimensions such as hour in the day, the distance is wrapped around, i.e., the distance between 23:00 and 1:00 is 2 hours, not 22 hours. We denote the two kernel functions as $k_\xi$ and $k_{P^d}$.

Given a particular auxiliary variable evaluation $\xi$, the conditional mean operator is computed as $\hat{\mu}_{P^d|\xi}=\sum_{i=1}^N \beta_i \phi(P^d_i)+\beta_{N+1} \phi(P^d_\xi)$ with $\beta$ computed following~\eqref{eq:beta}, where the last term is associated with the unknown $P^d$ under the query point $\xi$. The ambiguity set is defined following Proposition \ref{prop:amb}, and the conditional kernel DRO is solved in the dual form~\eqref{eq:dual_CDRO_emp}.
\begin{equation}\label{eq:dual_CDRO_OPF}
    \begin{aligned}
  \min_{\pref,\alpha\in\mathbb{R}^M,f_0\in \mathbb{R}} ~&f_0 + \sum_{i=1}^N \beta_i f(P^d_i)+\epsilon_x ||f||_\mathcal{G}\\
  \mathrm{s.t.}~& \text{OPF}(\pref,P^d_j)\le f(P^d_j)+f_0,~j=1,...,M,
  \end{aligned}
\end{equation}
where $f(P^d)=\sum_{j=1}^M \alpha_j l(P_j^d,P^d)$. Note that the objective depends on $f(P_i^d)$, which enumerates over the observation dataset ${\xi_i,P^d_i}_{i=1}^N$, whereas the constraints enumerates over $\{P_j^d\}_{j=1}^M$, which is the certification set. The certification set's role is to approximate the support of $P^d$, and we construct it by random sampling $P^d$ from the observation dataset with added Gaussian noise.

In practice, to reduce the computation complexity, we only consider the closest $m$ points in the observation dataset measured by the kernel distance. This is because points with a large distance to the query point in the RKHS will not influence the result much, thus can be ignored. The choice of $m$ depends on the kernel, i.e., how fast does $k(x,x')$ decrease to near zero as $x'$ moves away from $x$.

\subsection{Numerical experiments}\label{sec:result}

The case study uses the 11-bus ERCOT transmission system from \cite{battula2020ercot} with the historical hourly demand profile from ERCOT market information data hub \cite{ERCOT_EMIL}.
The corresponding historical weather data for the ERCOT zones was obtained from the National Solar Radiation Database \cite{NREL_NSRDB}.
All models in the case study are implemented using the CVXPY package \cite{diamond2016cvxpy} and the code and data are available in \href{https://github.com/chenyx09/CKDRO}{https://github.com/chenyx09/CKDRO}.

We consider three benchmarks.
\begin{itemize}
    \item The first one is the optimal cost, i.e., assuming that $P^d$ is perfectly known in the first stage, which is not realistic, yet it serves as the lower bound of the best cost attainable.
    \item The second benchmark is to solve the first stage as an OPF problem with the average $P^d$ in the dataset $\overline{P^d}$, i.e., replacing the expectation in~\eqref{eq:DRO_OPF} with $\overline{P^d}$. This setup ignores the auxiliary variable completely.
    \item The third benchmark is to perform a kernel interpolation with the same kernels used in DRO and perform OPF with the interpolated $P^d$. To be specific, given the auxiliary variable $x$, first compute $\beta$ following~\eqref{eq:beta}, then the interpolated $P_x^d$ is computed as $P_x^d=\sum_{i=1}^N \frac{\beta_i}{\overline{\beta}} P_i^d$ with $\overline{\beta}=\sum_{i=1}^N \beta_i$, then $\pref$ is obtained by solving the OPF with $P_x^d$.
\end{itemize}

The auxiliary variable has 18 dimensions, consisting of hours in the year (cyclic with period 8760, indicating the time in the year), hours in the day (cyclic with period 24, indicating time in the day), and temperature and precipitation in 8 regions.

\begin{figure}
    \centering
    \includegraphics[width=\columnwidth]{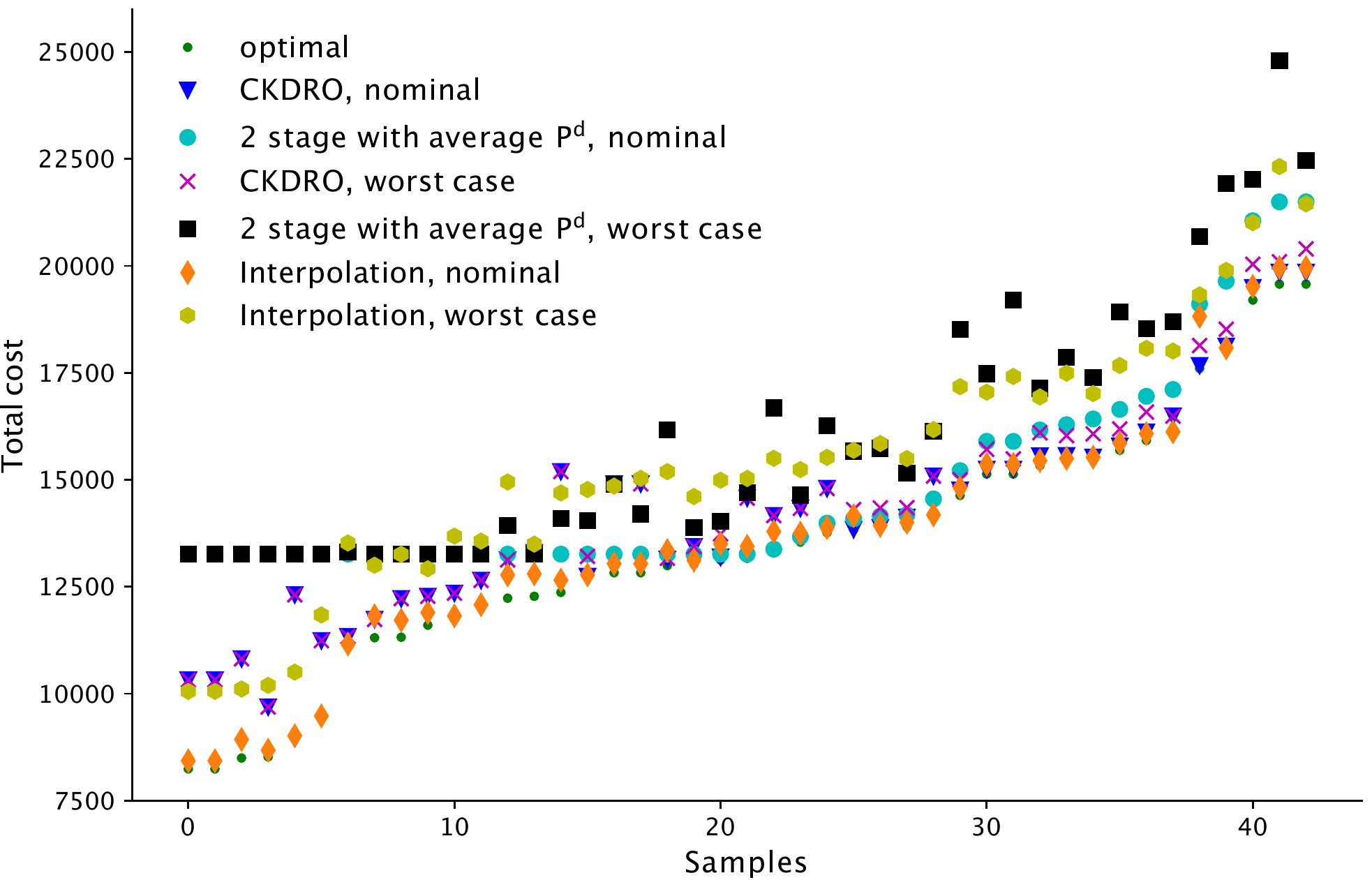}
    \caption{Comparison between CKDRO result and benchmarks for sample data set (50 points sorted by the objective value). The nominal cost of the CKDRO result is close to the optimal value whereas its worst case cost is very close to the nominal cost, showing greater robustness than the benchmarks.}
    \label{fig:comp_to_benchmark}
\end{figure}

We sample 60 points from the history dataset and for each data point $[\xi_i,P^d_i]$, the generation cost with the CKDRO solution is compared with the three benchmarks. To be specific, the CKDRO is solved with $\xi_i$, the two stage cost is then evaluated as  $\text{OPF}(\pref,P^d_i)$, where $\pref$ is the solution of the first stage CKDRO. For the three benchmarks, the first stage is solved as an OPF problem with $P_i^d$, $\overline{P^d}$ and $P^d_{\xi_i}$, respectively, then the two stage cost is evaluated with the solution from the first stage. Since the DRO setup aims at optimizing over the worst performance over the ambiguity set, in addition to the nominal two stage cost, we also plot the worst case cost, where Gaussian noise $w\sim\mathcal{N}(0,\sigma_w)$ is added to $P_i^d$ and the worst cost among $m$ samples of $\text{OPF}(\pref,P^d_i+w)$ is recorded.

\Cref{fig:comp_to_benchmark} shows the cost comparison between the CKDRO result and the benchmarks. The results are sorted with the optimal cost in ascending order. The kernel interpolation cost is almost identical to the optimal cost, showing that the kernel interpolation is able to predict $P_d^i$ fairly accurately. However, since it takes no uncertainty into account, its worst case cost is worse than the CKDRO result. The cost with the mean load profile $\overline{P^d}$ is significantly higher than the CKDRO result, and there is a big gap between its nominal cost and worst case cost, indicating that it is not robust to uncertainty. CKDRO is able to achieve close-to-optimal nominal cost while being much more robust to uncertainty, which is demonstrated by the little gap between its nominal cost and worst case cost.

\begin{figure}
    \centering
    \includegraphics[width=\columnwidth]{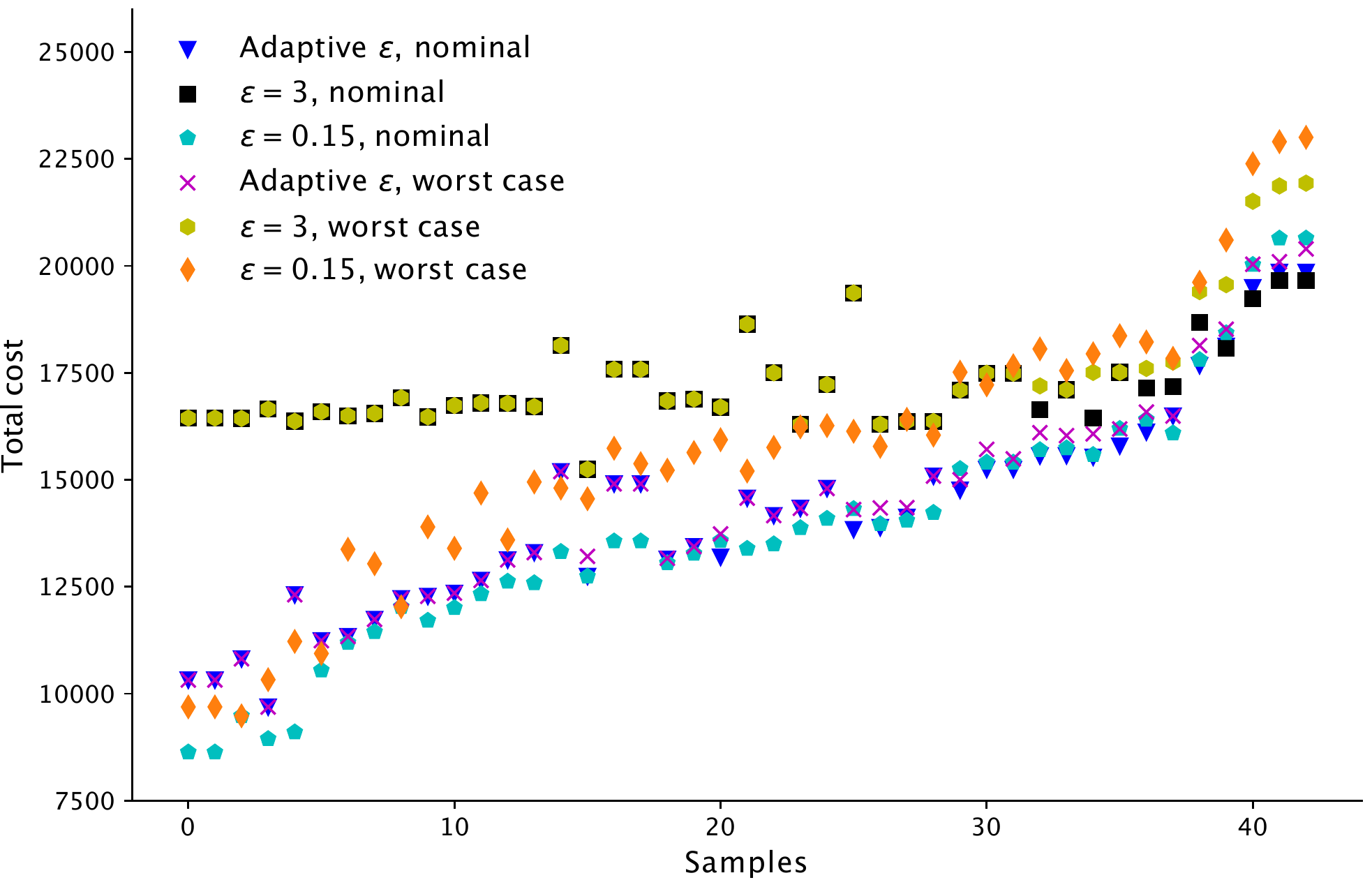}
    \caption{Influence of $\epsilon$ on the CKDRO result. A small $\epsilon$ results in a smaller nominal cost, yet the worst case cost can be huge; a large $\epsilon$ leads to a smaller gap between the nominal and worst case cost, yet can significantly hurt the nominal cost. The adaptive $\epsilon$ achieves the smallest worst case cost.}
    \label{fig:epsilon_case1}
\end{figure}

As shown in~\eqref{eq:dual_CDRO_OPF}, $\epsilon$ determines the level of robustness for the CKDRO, and changes with the number of data points and the location of the auxiliary variable. To showcase the effect of $\epsilon$, \cref{fig:epsilon_case1} shows the comparison of the CKDRO result with the adaptive $\epsilon$ and two benchmarks, one fixing $\epsilon=0.15$, and the other fixing $\epsilon=3$. A smaller $\epsilon$ causes the DRO to consider distributions that are more concentrated around the conditional mean, resulting in better nominal cost, yet may struggle when $P^d$ deviates from the conditional mean, which is shown in the worst case cost. A large $\epsilon$, on the other hand, leads to a smaller gap between the nominal and worst case cost, sometimes even 0 gap, yet significantly hurt the nominal performance. The adaptive $\epsilon$ achieves the best worst case performance, as is the goal of the DRO.

\section{Conclusion}\label{sec:conclusion}
We proposed a conditional kernel DRO framework capable of robust decision making that leverages the conditional distribution of the unknown variable. The main application of CKDRO is when an auxiliary variable is observable and correlated with the unknown variable, and one needs to make a decision after observing the auxiliary variable. The key idea is to use the conditional mean operator in RKHS space to define the ambiguity set, whose size depends on the data volume as well as the query point position. We draw connections between the ambiguity set used in CKDRO and  one defined with the Wasserstein distance and show that they are related via the change of test function. The CKDRO is solved in its dual form and approximated by a finite-dimensional convex program. We show its application on a optimal power flow problem and the results show that CKDRO is able to outperform several common benchmarks by leveraging the auxiliary variable and the uncertainty in conditional distribution.
\balance
\bibliographystyle{myieeetran}
\bibliography{mybib}
\end{document}